\documentclass{amsart}
\usepackage[all]{xy}
\usepackage{amssymb}

\usepackage{tikz}
\usepackage{color}
\usepackage{comment}
\usepackage{graphicx}
\usepackage{float}
\usepackage{xcolor}
\usepackage{xspace}
\usepackage{enumerate}

\definecolor{ruta2}{rgb}{0.409, 0.459, 0.208}
\definecolor{vino}{rgb}{0.4,0,0}
\definecolor{ruta}{rgb}{0.153,0.076,0.0}

\let\cal\mathcal
\def\Ascr{{\cal A}}

\def\Escr{{\cal E}}
\def\Fscr{{\cal F}}
\def\Gscr{{\cal G}}
\def\Hscr{{\cal H}}

\def\Kscr{{\cal K}}
\def\Lscr{{\cal L}}
\def\Mscr{{\cal M}}
\def\Nscr{{\cal N}}
\def\Oscr{{\cal O}}

\def\Tscr{{\cal T}}

\def\Xscr{{\cal X}}
\def\Yscr{{\cal Y}}

\let\blb\mathbb

\def \ZZ{{\blb Z}}

\def \NN{{\blb N}}

\def\quot{/\!\!/}

\def\Mod{\operatorname{Mod}}

\def\Qch{\operatorname{Qch}}
\def\coh{\mathop{\text{\upshape{coh}}}}

\def\Spec{\operatorname {Spec}}

\def\Ext{\operatorname {Ext}}
\def\Hom{\operatorname {Hom}}
\def\uHom{\operatorname {\mathcal{H}\mathit{om}}}
\def\uEnd{\operatorname {\mathcal{E}\mathit{nd}}}
\def\End{\operatorname {End}}
\def\RHom{\operatorname {RHom}}
\def\uRHom{\operatorname {R\mathcal{H}\mathit{om}}}

\def\coker{\operatorname {coker}}

\def\End{\operatorname {End}}

\def\Pic{\operatorname {Pic}}

\def\gldim{\operatorname {gl\,dim}}

\def\r{\rightarrow}

\DeclareMathOperator{\Proj}{Proj}

\def\Ehr{\operatorname {Ehr}}
\def\Vol{\operatorname {Vol}}

\newcommand{\s}{\sigma}
\newcommand{\la}{\langle}
\newcommand{\ra}{\rangle}

\def\Xs{X^{\mathbf{s}}}

\newtheorem{lemma}{Lemma}[section]
\newtheorem{proposition}[lemma]{Proposition}
\newtheorem{theorem}[lemma]{Theorem}
\newtheorem{corollary}[lemma]{Corollary}

\theoremstyle{definition}

\newtheorem{example}[lemma]{Example}
\newtheorem{definition}[lemma]{Definition}

{

}

\theoremstyle{remark}

\newtheorem{remark}[lemma]{Remark}

\newdimen\uboxsep \uboxsep=1ex
\def\uboxn#1{\vtop to 0pt{\hrule height 0pt depth 0pt\vskip\uboxsep
\hbox to 0pt{\hss #1\hss}\vss}}

\def\uboxs#1{\vbox to 0pt{\vss\hbox to 0pt{\hss #1\hss}
\vskip\uboxsep\hrule height 0pt depth 0pt}}

\def\codim{\operatorname{codim}}

\def\Stab{\operatorname{Stab}}
\let\oldmarginpar\marginpar
\def\marginpar#1{ \oldmarginpar{\tiny #1}}

\def\Sym{\operatorname{Sym}}
\title{Non-commutative crepant resolutions for some toric singularities I}
\author{\v{S}pela \v{S}penko}
\address{(\v{S}pela \v{S}penko) {\tt spela.spenko@ed.ac.uk}\\ School of Mathematics\\
  The University of Edinburgh\\
  James Clerk Maxwell Building\\
  The King's Buildings\\
  Peter Guthrie Tait Road\\
  EDINBURGH\\
  EH9 3FD \\
  Scotland, UK} \author{Michel Van den Bergh} 
\address{(Michel Van den Bergh) {\tt michel.vandenbergh@uhasselt.be}\\ Department of mathematics\\Universiteit Hasselt\\
  Martelarenlaan 42\\
  3500 Hasselt\\
  Belgium} 
\thanks{The first author is supported by EPSRC grant EP/M008460/1.}
\thanks{The second author is a senior researcher at the
  Research Foundation Flanders (FWO). While working on this project he was supported by
the FWO grant G0D8616N: ``Hochschild cohomology and deformation theory of triangulated categories''.}
\begin{document}
\begin{abstract}
We give a criterion for the existence of non-commutative crepant resolutions (NCCR's) for certain toric  singularities.
In particular we recover Broomhead's result that a 3-dimensional toric Gorenstein singularity
has an NCCR. Our result also yields the existence of an NCCR for a 4-dimensional toric Gorenstein singularity which
is known to have no toric NCCR.
\end{abstract}
\maketitle
\section{Introduction}
In this note we discuss the existence of non-commutative crepant resolutions (NCCRs) for some toric singularities. 
Let us first recall the definition. For the rationale behind the definition of an NCCR see \cite{VdB32,Leuschke}. 
Throughout $k$ is an algebraically closed field of characteristic zero. 
\begin{definition}\cite{DaoIyama,Leuschke,SVdB,VdB32,Wemyss1}
Let $R$ be a normal Gorenstein 
 domain. A \emph{non-commutative crepant resolution} (NCCR) of $R$
is an $R$-algebra of finite global dimension 
   of the form $\Lambda=\End_R(M)$  which in addition is Cohen-Macaulay as $R$-module and where $M$ is a non-zero finitely generated 
reflexive
$R$-module.
\end{definition}

 The following proposition is a combination of our main results. 
For a representation $X$ of a reductive group $G$ we denote by $X^u:=\{x\in X\mid 0\in \overline{Gx}\}$ the unstable locus. 
\begin{proposition}[\S \ref{combi}]
\label{prop:main}
Let $W$ be a generic  unimodular representation of an abelian reductive group $G$ over $k$, and let  $X:=\Spec \Sym(W)=W^\vee$. 
If  
$\dim \hbox{{$X^u$}}-\dim G\le 1$ 
then  $\Sym(W)^G$ has an NCCR.
\end{proposition}
For the definition of a generic, unimodular representation see 
 Definitions \ref{def:generic}, \ref{def:unimodular}, respectively. Recall that an abelian reductive group over $k$ is a product of a
torus and a finite abelian group.

\medskip

Proposition \ref{prop:main} gives a relatively easy proof that three-dimensional toric Gorenstein singularities have an NCCR (see Corollary \ref{cor:nccr2}), a fact first proved 
by Broomhead \cite{Broom}. Actually Broomhead establishes the existence of a ``toric'' \cite{Bocklandt} NCCR ($M$ is a sum of reflexive ideals)
which is much more difficult and relies on the theory of dimer models. In \cite{SVdB5}  we give an alternative proof of Broomhead's result which is however still not easy.

\medskip

In \cite[\S10.1]{SVdB} we constructed toric NCCRs for toric rings coming from quasi-symmetric representations $W$ (e.g. self-dual), and showed that in general toric NCCRs do not always exist. 
In other words, Broomhead's result does not extend to higher dimension.
In fact in \cite[\S10.1]{SVdB} we gave an example of a 4-dimensional toric Gorenstein singularity which does not 
have a toric NCCR. Using Proposition \ref{prop:main} we can now show that it nevertheless has a non-toric NCCR. See Example \ref{ex:toric} below. 
On the other hand,  Higashitani and Nakajima \cite{Nakajima} recently constructed toric NCCRs for some natural examples of toric rings not coming from quasi-symmetric representations.

\section{Acknowledgment}
We thank the referee for his careful reading of the paper and many useful comments which considerably improved its exposition.

\section{Notation and conventions}
All
objects are defined over $k$. If $\Xscr$ is a stack then we write $D(\Xscr)$ for the unbounded derived category $D_{\Qch}(\Mod(\Oscr_\Xscr))$ of $\Oscr_\Xscr$-modules with quasi-coherent cohomology. 

\medskip

For a reductive group $G$ we denote by $X(G)$
 (resp. $Y(G))$ the character group (resp. the group of one-parameter subgroups) of $G$. 
There is a natural pairing $Y(G)\times X(G)\to \ZZ$, we denote it by $\la\,,\,\ra$.

If a reductive group $G$ acts on an affine variety $X$ and $\chi\in X(G)$ is a character then
we write $X^{ss,\chi}$ for the open subset of $X$ consisting of the $\chi$-semi-stable points in $X$.
In other words (following \cite{King}), $X^{ss,\chi}$ consists of the points $x\in X$ such that for
$\lambda\in Y(G)$ of $G$ with the property that $\lim_{t\r 0} \lambda(t)x$ exists then $\langle \lambda,\chi\rangle \ge 0$.
We say that $x\in X$ is stable if it has a closed orbit and a finite stabilizer. We denote the locus of stable points in $X$ by $X^s$. We have $X^s\subset X^{ss,\chi}$ for
any $\chi$. Moreover, we write $X^{\mathbf{s}}\subset X$ for the set of points with closed orbit and trivial stabilizer.

We write $X^u=\{x\mid 0\in \overline{Gx}\}$ for the $G$-unstable locus or nullcone.

The inclusions between the  open subschemes of $X$ that were introduced are summarized in the following diagram 
\begin{equation*}
\xymatrix{
X^{\mathbf{s}}\ar@{^(->}[r] & X^{s}\ar@{^(->}[r]\ar@{^(->}[d] & X^{ss,\chi}\ar@{^(->}[d]\\
& X\setminus X^u\ar@{^(->}[r] & X\\
}
\end{equation*}

\begin{definition}
\label{def:generic}
We say that a reductive group $G$ acts \emph{generically} on a smooth affine  variety $X$ if
$\codim  (X-\Xs,X) \ge 2$.
If $W$ is a $G$-representation then we  say that $W$ is \emph{generic} if $G$ acts generically on
$\Spec \Sym(W)\cong W^\vee$.
\end{definition}

\begin{definition}\label{def:unimodular}
Let $W$ be $d$-dimensional representation of an algebraic group $G$. 
We say that $W$ is a {\em unimodular} if $\wedge^d W\cong k$, where $k$ is the trivial representation.
\end{definition}

A variety is an integral separated scheme of finite type over $k$. 
If $X$ is a variety with an action of a reductive group $G$ then we consider a $G$-equivariant sheaf on $X$ as a sheaf on the stack $X/G$.


\section{Main result}

The next theorem extends \cite[Thm.\ 5.1]{VdB32} to certain Deligne-Mumford stacks \cite{DM,Champs,Olsson}. As a consequence we obtain NCCRs  (see Corollary \ref{cor:NCCR}). 
The proof of the theorem is given in \S\ref{sec:proof}. 

Let us recall that if $X$ is an affine variety then $X\quot G=\Spec k[X]^G$, $X^{ss,\chi}\quot G=\Proj \Gamma_*(X)^G$ where $\Gamma_*(X)=\bigoplus_n \Gamma(X,\chi^{-n}\otimes \Oscr_X)$. 
Note that $\Gamma_0(X)=k[X]^G$ and the inclusion $\Gamma_0(X)\hookrightarrow\Gamma_*(X)$ defines a natural projective map $\theta:X^{ss,\chi}\quot G\r X\quot G$.
\begin{theorem}
\label{th:dim1}
Let $G$ be an abelian reductive group over $k$  and let $X$ be a smooth affine
$G$-variety containing a $G$-stable point.  
  Let $\chi\in
X(G)$ be a character such that every point in $X^{ss,\chi}$ has finite
stabilizer (i.e.\ $X^{ss,\chi}/G$ is a Deligne-Mumford stack) and
assume in addition that $\theta:X^{ss,\chi}\quot G\r X\quot G$ has fibers of
dimension $\le 1$. Then $\coh(X^{ss,\chi}/G)$ contains an object
$\Tscr$ with the following properties.
\begin{enumerate}
\item \label{th:dim1:a} $\Tscr$ is a vector bundle on $X^{ss,\chi}$.
\item \label{th:dim1:b} $\Ext^i_{X^{ss,\chi}/G}(\Tscr,\Tscr)=0$ for $i>0$.
\item \label{th:dim1:c} $\Tscr$ is a generator\footnote{See Definition \ref{def:generator} for the definition of the generation.} for $D(X^{ss,\chi}/G)$.  
\end{enumerate}
\end{theorem}
An object satisfying \eqref{th:dim1:a}\eqref{th:dim1:b}\eqref{th:dim1:c} is
sometimes called a \emph{tilting bundle}. 
\begin{remark} \label{rem:dm}
If $G$ is an
  abelian reductive group acting linearly on an affine variety then
  $X^{ss,\chi}/G$ will be a Deligne-Mumford stack if $\chi\in X(G)$ is
  chosen  generically. Indeed we may choose a closed embedding of $X$
in a $G$-representation and for a representation the claim follows from \cite[Theorem 14.3.14]{CoxLittleSchenck} 
(see also  \cite[Proposition 2.1]{HLSam}).
\end{remark}

Whenever we are in the setting of Theorem
  \ref{th:dim1} we will use the following diagram 
\begin{equation}
\label{eq:maps}
\xymatrix{
X^{ss,\chi}\ar@{^(->}[r]^{\tilde{\theta}}\ar[d]\ar@/_3em/[dd]_{\pi} & X\ar[d]\ar@/^3em/[dd]^{\gamma}\\
X^{ss,\chi}/G\ar@{^(->}[r]\ar[d]^{\pi_s} & X/G\ar[d]_{\gamma_s}\\
X^{ss,\chi}\quot G\ar[r]_{\theta} & X\quot G
}
\end{equation}
where $\tilde{\theta}$ is an inclusion, $\theta$ the induced map on the quotients (coming from the definition of the quotients), $\pi$, $\gamma$ are quotient maps, and $\pi_s$, $\gamma_s$ are stack morphisms. More precisely, the morphism $\gamma:X\to X\quot G$ is $G$-equivariant 
and hence it factors through $X/G$ which yields $\gamma_s$. 
A similar statement holds for $\pi_s$. 

Under some genericity conditions (in the sense of Definition \ref{def:generic}) one may obtain an NCCR from Theorem \ref{th:dim1}. We denote by $R=k[X]^G$ the coordinate ring of $X\quot G$.

\begin{corollary}
\label{cor:dim1} Let $X,G, \chi,\Tscr$ be as in Theorem
  \ref{th:dim1}. Then $D(X^{ss,\chi}/G)\cong D(\Lambda)$ where
  $\Lambda=\End_{X^{ss,\chi}/G}(\Tscr)$. One has $\gldim\Lambda< \infty$.
 Moreover, if 
  $G$ acts generically on $X$ 
  then $\Lambda=\End_R(T)$ 
  where  
  $T=\Gamma(X^{ss,\chi}/G,\Tscr)=\Gamma(X^{ss,\chi},\Tscr)^G$ which is a reflexive $R$-module. 
\end{corollary}
\begin{proof} The derived equivalence claim follows from \cite[Theorem 4.3]{Keller1}. The derived equivalence implies  $\gldim\Lambda\le \infty$
since $X^{ss,\chi}/G$ is smooth (see \cite[Theorem 7.6]{HilleVdB}). 
We now refer to \cite[\S3,4]{SVdB} for some generalities concerning
reflexive sheaves we use below. Recall in particular that reflexive sheaves $\Fscr$, $\Gscr$
on a normal variety $Z$ form a rigid monoidal category with tensor product
$\Fscr\otimes \Gscr:=(\Fscr\otimes_Z\Gscr)^{\vee\vee}$.
Assume that $\codim (X-X^{\bf s},X)\ge 2$. Then also 
$\codim (X-X^{ss,\chi},X)\ge 2$ and hence $\tilde{\theta}_\ast$ defines a monoidal
equivalence between the categories of reflexive sheaves on $X^{ss,\chi}$
and $X$. Using again the condition $\codim (X-X^{\bf s},X)\ge 2$,
taking $G$-invariants defines a monoidal equivalence between $G$-equivariant
reflexive sheaves on $X$ and reflexive sheaves on $X\quot G$ by \cite[Lemma 4.1.3]{SVdB5}.  
Therefore $T=\Gamma(X^{ss,\chi},\Tscr)^G=\Gamma(X,\tilde{\theta}_\ast\Tscr)^G=\Gamma(X\quot G,(\gamma_*\tilde{\theta}_\ast\Tscr)^G)$, and in particular $T$ is a reflexive $R$-module. 
Again using the mentioned monoidal equivalences we obtain 
$\Lambda=\End_{X^{ss,\chi}/G}(\Tscr)=\End_{X\quot G}((\gamma_*\tilde{\theta}_\ast\Tscr)^G)=\End_R(T)$. 
\end{proof}

\begin{corollary} \label{cor:NCCR} 
Let $X,G,\chi$ be as in Theorem \ref{th:dim1}.  Assume in addition that
$X=W^\vee$ where $W$ is generic unimodular $G$-representation.
Then $R=\Sym(W)^G$ has an NCCR.
\end{corollary}
\begin{proof}
Let
$\Ascr=\uEnd_{X^{ss,\chi}/G}(\Tscr)$ where $\Tscr$ is as in Theorem \ref{th:dim1}.
Then $\Ascr$ is a sheaf of algebras on $X^{ss,\chi}/G$. 
By Corollary \ref{cor:dim1} we have to show that $\Lambda=Rf_{s,\ast} \Ascr$
is Cohen-Macaulay. 
Using Lemma \ref{lem:crep} below we have 
by  the same argument as \cite[Lemma 3.2.9]{VdB04a} noting that $f_s=\theta\circ \pi_s$ is proper (since $\pi_s$, $\theta$ are proper by  \cite[Exercise 11.E]{Olsson}, \cite[Proposition 14.1.12]{CoxLittleSchenck}, resp.) and referring to \cite[Corollary 2.10]{Nironi} for the first equality
\begin{align*}
\RHom_{X\quot G}(Rf_{s,\ast} \Ascr,\omega_{X\quot G})&=\RHom_{X^{ss,\chi}/G}(\Ascr,f_s^!\omega_{X\quot G})\\
&=\RHom_{X^{ss,\chi}/G}(\Ascr,\omega_{X^{ss,\chi}/G})\\
&=\RHom_{X^{ss,\chi}/G}(\Ascr,\Oscr_{X^{ss,\chi}/G})\\
&=Rf_{s,\ast}\Ascr^\vee\\
&= Rf_{s,\ast}\Ascr\\
&=f_{s,\ast}\Ascr
\end{align*}
This finishes the proof.
\end{proof}

We have used the following lemma.

\begin{lemma}\label{lem:crep}
Let $X,G,\chi$ be as in Theorem \ref{th:dim1}. 
Assume in addition that
$X=W^\vee$ where $W$ is a generic unimodular $G$-representation. 
The map $f_s=\theta\circ\pi_s$ is crepant and  $\omega_{X^{ss,\chi}/G}\cong \Oscr_{X^{ss,\chi}/G}$. 
\end{lemma}

\begin{proof}
The hypothesis imply that $\omega_{X\quot G}$ is invertible and moreover $\omega_{X\quot G}\cong \Oscr_{X\quot G}$ by \cite[Satz 2]{Knop2} because of the unimodularity. 
 A Deligne-Mumford stack is \'etale locally a quotient stack for a finite group and in particular $\omega_{X^{ss,\chi}/G}$ is a reflexive sheaf (it is in fact invertible but already reflexivity suffices our purposes). 
We claim $f_s^\ast\omega_{X\quot G}=\omega_{X^{ss,\chi}/G}$ and hence in particular $\omega_{X^{ss,\chi}/G}\cong \Oscr_{X^{ss,\chi}/G}$. 
This follows from the fact
both $\omega_{X\quot G}$ and $\omega_{X^{ss,\chi}/G}$ are reflexive 
and $f_s$ is the identity on $X^{\bf s}/G\cong X^{\bf s}\quot G$, using that  the complement of $X^{\bf s}$ in $X$ is of codimension $\geq 2$ by the genericity assumption.
\end{proof}

\begin{remark}
The assumption that $W$ is  generic simplifies the proof of the previous lemma but it is in fact superfluous. 
This is a consequence of the theory of toric DM stacks~\cite{BorisovHorja}. 
See e.g. \cite[Lemma A.2]{SVdB5}.
\end{remark}

\section{Proof of Theorem \ref{th:dim1}}\label{sec:proof}
 We refer to \cite[Definition 3.3.1]{SVdB3} for the definition of a good quotient. 
Assume $Y$ is such that a good quotient $Y\quot G$ exists (in particular $G$ is reductive).
For
an open $U\subset Y\quot G$ we write $\tilde{U}=U\times_{Y\quot G} Y\subset Y$. 

\begin{definition}\label{def:generator}
Let $\Yscr$ be a stack. 
We say that a derived category $D(\Yscr)$ is  generated by an object $E\in D(\Yscr)$ if $E^\perp=0$.  
We say that $D(\Yscr)$ for $\Yscr=Y/G$ such that a good quotient $Y\quot G$ exists is \emph{locally generated}
by a perfect object $E$ if $D(\tilde{U}/G)$ is generated by $E{|}\tilde{U}$ for every affine open
$U\subset Y\quot G$, i.e. $(E{|}\tilde{U})^\perp=0$.
\end{definition}

\begin{remark}\label{rem:gen}
From the fact that $G$-equivariant complexes on $\tilde{U}$ can be extended to complexes on $Y$ (for example by pushforward), it follows that $E$ being a  local generator is 
equivalent to the statement that $R\uHom_{Y/G}(E,\Fscr)=0$ implies $\Fscr=0$.
\end{remark}

The following is a variant on \cite[Lemma 3.5.4]{SVdB3}. It can be deduced from a more general result (see \cite[Lemma 1.3, Theorem 5.7]{OS}). However it seems useful to give a direct proof in our simple setting. 

\begin{lemma} \label{lem:local}
Let $G$ be a reductive group acting on an algebraic variety $Y$ such that
a good quotient
$\pi:Y\rightarrow
Y\quot G$ exists and such that $Y/G$ is a Deligne-Mumford stack. Then $D(Y/G)$ is {locally generated} by $V\otimes \Oscr_Y$
for a single finite dimensional representation $V$ of $G$.
\end{lemma}
\begin{proof} 
By Remark \ref{rem:gen},  we need to find $V$ such that $\pi_{s,\ast}R\uHom_{Y/G}(V\otimes \Oscr_Y,\Fscr)=0$ 
implies $\Fscr=0$, where $\pi_s:Y/G\r Y\quot G$ is the morphism of stacks associated to
$\pi$. Since $\pi_{s,\ast}$ is exact, as $\pi$ is a good quotient, and $V\otimes \Oscr_Y$ is a vector bundle we have $H^*(\pi_{s,\ast}R\uHom_{Y/G}(V\otimes \Oscr_Y,\Fscr))=\pi_{s,\ast}\uHom_{Y/G}(V\otimes \Oscr_Y,H^*(\Fscr))$. In other words, it is sufficient to prove 
that $\pi_{s,\ast}\uHom_{Y/G}(V\otimes \Oscr_Y,\Fscr)=0$ 
implies $\Fscr=0$
  for $\Fscr\in \Qch(Y/G)$. 

If a certain $V$ works then any representation containing  $V$ works
as well. Hence we claim that the existence of $V$ is a local property for the \'etale topology on $Y\quot G$. 
Let $(U_i\to Y\quot G)_i$ be an \'etale covering of $Y\quot G$. 
Let $\tilde{U}_i=Y\times_{Y\quot G} U_i$. 
Since $G$ is reductive (and $\pi$ is a good quotient) one can see that $U_i=\tilde{U}_i\quot G$.\footnote{It is easy to see that a good quotient (in the sense of \cite[Definition 3.3.1]{SVdB4}) is compatible with arbitrary base extension; i.e. if $Y\to X$ is a good quotient and $Z\to X$ is arbitrary, then $Y\times_X Z\to Z$ is also a good quotient. To see this note that $Y\to X$ is built by gluing morphisms $\Spec A\to \Spec A^G$ and this allows us to reduce to the affine case. Let $Y=\Spec A$, $X=\Spec A^G,$ $Z=\Spec B$. Then the dual statement $B=(A\otimes_{A^G} B)^G$ holds since the inclusion $A^G\hookrightarrow A$ is split by the Reynolds operator.}
 We denote $\pi^i_{s}:\tilde{U}_i/G \to U_i\quot G$. 
 Let us assume that for every $i$ there exists $V_i$ such that 
$\pi^i_{s,*}\uHom_{\tilde{U}_i/G}(V_i\otimes \Oscr_{\tilde{U}_i},\Fscr)=0$ implies $\Fscr=0$ for $\Fscr\in \Qch(\tilde{U}_i/G)$. 
As $Y\quot G$ is quasi-compact (and as an \'etale map is open) we only need a finite number of $U_i$ such that $(U_i\to Y\quot G)_{i=1}^n$ is an \'etale covering.  Let $V=\oplus_{i=1}^n V_i$. Assume that $\Hscr=\pi_{s,*}\uHom_{Y/G}(V\otimes \Oscr_Y,\Fscr)=0$. We need to prove that $\Fscr=0$. 
Let us write $\Gscr_i$ for the pullback of $\Gscr\in \Qch(Y/G)$ to $\tilde{U}_i/G$.
The restriction $\Hscr_i$ to $\tilde{U_i}/G$ is $0$. 
Moreover, flatness of \'etale morphisms implies that $\Hscr_i=\pi^i_{s,*}\uHom_{\tilde{U}_i/G}(V_i\otimes \Oscr_{\tilde{U}_i},\Fscr_i)=0$. 
Thus, $\Fscr_i=0$ by our assumption, and hence $\Fscr=0$.

We may therefore assume that
$Y$ is affine, and furthermore it suffices to show that $\Fscr$ is zero in a neighbourhood of any closed orbit by \cite[Lemma 4.4.3]{SVdB3}. 
Invoking the Luna slice theorem we may assume that  $\pi$
is of the form $G\times^H S\r (G\times^H S)\quot G\cong S\quot H$ where $S$ is an \'etale slice at $y\in Y$ with
closed orbit and $H=\Stab(y)$. Since $Y/G$ is a Deligne-Mumford stack, $H$ is finite. 
Let $kH$ be the regular $H$-representation. Then
$\uHom_{S/H}(kH\otimes \Oscr_S,\Fscr)=0$ implies $\Fscr=0$. Since $S/H\cong (G\times^H S)/G$, $kH\otimes_k \Oscr_S$ corresponds to a $G$-equivariant vector bundle $\Escr$ on $G\times^H S$. It now
suffices, using the reduction to $Y=G\times^H S$ and $Y/G\cong S/H$, to write $\Escr$ as a quotient of $V\otimes \Oscr_{G\times^H S}$ for some finite
dimensional $G$-representation $V$.
\end{proof}
\begin{lemma} \label{lem:multiple} Let $G$ be a reductive group acting on an algebraic variety~$Y$ which is projective over an affine variety and let $\Mscr$ be an ample $G$-equivariant line bundle on $Y$. Let $Y^{ss}\subset Y$ be the semi-stable
locus corresponding to the linearization given by $\Mscr$ and let $\pi:Y^{ss}\r Y^{ss}\quot G$ be the (good)\footnote{See e.g. \cite[\S3.4]{SVdB3}.} quotient map. Then up to replacing $\Mscr$ by a strictly positive multiple we may assume that
$(\Mscr|_{Y^{ss}})^G$ is an ample line bundle on $Y^{ss}\quot G$ generated by global sections such that moreover $\pi^\ast ((\Mscr|_{Y^{ss}})^G)\cong\Mscr|_{Y^{ss}}$.
\end{lemma}
\begin{proof} Put
\[
\Gamma_\ast(Y)=\bigoplus_{n\ge 0} \Gamma(Y,\Mscr^{\otimes n}).
\]
Then $Y^{ss}\quot G=\Proj \Gamma_\ast(Y)^G$. Since $\Gamma_\ast(Y)$ and $\Gamma_\ast(Y)^G$ are finitely generated, there is an $N$ such that the $N$'th  Veronese subalgebras of $\Gamma_\ast(Y)$ 
and $\Gamma_\ast(Y)^G$ are both generated in degree one. We then replace $\Mscr$ by $\Mscr^{\otimes N}$. 
\end{proof}
\begin{lemma} \label{lem:compact}  Let $G$ be a reductive group acting on an algebraic variety~$Y$
which is
  projective over an affine variety and let $\Mscr$ be an ample
  $G$-equivariant line bundle on $Y$. Let $Y^{ss}\subset Y$ be the
  semi-stable locus corresponding to the linearization given by
  $\Mscr$ and let $\pi:Y^{ss}\r Y^{ss}\quot G$, $\pi_s:Y^{ss}/G\r Y^{ss}\quot G$ 
be the associated quotient maps. We assume that $Y^{ss}/G$ is a Deligne-Mumford stack.

In addition we assume that we have replaced $\Mscr$ by a strictly positive multiple such that $\Lscr:=(\Mscr\mid Y^{ss})^G\in \Pic(Y^{ss}\quot G)$ has the properties exhibited in Lemma \ref{lem:multiple}. Put $Z=\Spec \Gamma(Y,\Oscr_Y)$.
Let $d$ be the maximum of
the dimension of the fibers of $Y^{ss}\quot G\r Z\quot G$. Let $V$ be a finite dimensional representation of $G$ such that $V\otimes \Oscr_{Y^{ss}}$ is a local generator for $D(Y^{ss}/G)$
as in Lemma \ref{lem:local}. Put $V=\bigoplus_{i=1}^n V_i$ with $V_i$ irreducible and fix $m_i\in \ZZ$ for $i=1,\dots,n$ and $l\ge 1$.
Then 
\[
\bigoplus_{j=0}^d \bigoplus_{i=1}^{n} V_i\otimes \pi_s^{\ast}(\Lscr)^{\otimes lj+m_i}
=\bigoplus_{j=0}^d \bigoplus_{i=1}^{n} V_i\otimes \Mscr^{\otimes lj+m_i}\biggr|_{Y^{ss}}
\]
is a compact generator for $D(Y^{ss}/G)$.
\end{lemma}
\begin{proof} Replacing $\Mscr$ by $\Mscr^{\otimes l}$ we may assume $l=1$.
Put $\Escr=\bigoplus_{i=1}^{n} V_i\otimes \pi_s^{\ast}(\Lscr)^{\otimes m_i}$. Then since $\pi_s^{\ast}(\Lscr)$
is locally free on $Y^{ss}/G$, 
$\Escr$ is a local
generator for $D(Y^{ss}/G)$. We must prove that $\bigoplus_{j=0}^d \Escr\otimes 
\pi_s^{\ast}(\Lscr)^{\otimes j}$ is a generator for $D(Y^{ss}/G)$.

Assume $\Fscr\in D(Y^{ss}/G)$ is such that $\RHom_{Y^{ss}/G}(\bigoplus_{j=0}^d \Escr\otimes 
\pi_s^{\ast}(\Lscr)^{\otimes j},\Fscr)=0$. Then $\RHom_{Y^{ss}\quot G}(\bigoplus_{j=0}^d \Lscr^{\otimes j} ,\pi_{s,\ast} \uRHom_{Y^{ss}/G}(\Escr,\Fscr))=0$. By \cite[Lemma 3.2.2]{VdB04a} this implies
$\pi_{s,\ast} \uRHom_{Y^{ss}/G}(\Escr,\Fscr)=0$. Since $\Escr$ is a local generator this implies $\Fscr=0$.
\end{proof}

\begin{lemma} \label{lem:vanishing}
Let $X,G,\chi$ be as in Theorem \ref{th:dim1}. 
Then $\theta$ is birational and it is true that $R\theta_\ast \Oscr_{X^{ss,\chi}\quot G}=\Oscr_{X\quot G}$. 
Finally $R^i\theta_\ast=0$ for $i>1$.
\end{lemma}
\begin{proof} Both $X^{ss,\chi}\quot G$ and $X\quot G$ contain $X^s\quot G$
as an nonempty hence dense subscheme. So they
are birational.
Both $X^{ss,\chi}\quot G$ and $X\quot G$ are quotients by
  reductive groups and hence they have rational singularities (see \cite[Corollaire]{Boutot}). 
This proves the claim about $R\theta_\ast\Oscr_{X^{ss,\chi}\quot G} $. The last claim follows from
the hypothesis that the fibers of $\theta$ have dimension $\le 1$.
\end{proof}

\begin{lemma} \label{lem:generator}
Let $G,X,\chi$ be as in Theorem \ref{th:dim1}. Then there exist characters $(\chi_u)_{u=1,\ldots,N}$ such that $\Lscr_u=\chi_u\otimes \Oscr_{X^{ss,\chi}}$ for $i=1,\ldots, N$ generate $D(X^{ss,\chi}/G)$
and such that moreover we
have
\begin{equation}
\label{eq:ext_cond1}
\Ext^i_{X^{ss,\chi}/G}(\Lscr_u,\Lscr_u)=0\qquad\text{ for $i>0$}
\end{equation}
\begin{equation}
\label{eq:ext_cond2}
\Ext^i_{X^{ss,\chi}/G}(\Lscr_u,\Lscr_v)=0\qquad\text{ for $i>1$}
\end{equation}
\begin{equation}
\label{eq:ext_cond3}
\Ext^1_{X^{ss,\chi}/G}(\Lscr_u,\Lscr_v)=0\qquad\text{ for $u<v$}
\end{equation}
\end{lemma}

\begin{proof} 
According to Lemma \ref{lem:compact} (with $Y=X$, $\Mscr=\chi\otimes \Oscr_{X^{ss,\chi}}$) after replacing $\chi$ by some strict positive multiple there exist $\mu_i\in X(G)$ (corresponding to the character of $V_i=\mu_i\otimes \Oscr_{X^{ss,\chi}}$ in Lemma \ref{lem:compact}) such that 
for any collection of $m_i\in \ZZ$ and for any $\ell\geq 1$ the object 
\[
\bigoplus_{j=0}^1 \bigoplus_{i=1}^{n} \mu_i\otimes \chi^{lj+m_i} \otimes \Oscr_{X^{ss,\chi}}
\]
is a compact generator of $D(X^{ss,\chi}/G)$.  
We put $\chi_1=\mu_1\otimes \chi^{m_1}$, $\chi_2=\mu_1\otimes \chi^{l+m_1}$, $\chi_3=\mu_2\otimes \chi^{m_2}$, $\dots$, $\chi_{2n}=\mu_n\otimes \chi^{l+m_n}$, $N=2n$.
Then  \eqref{eq:ext_cond1}, \eqref{eq:ext_cond2} follow directly from Lemma \ref{lem:vanishing} (because $\pi_{s,\ast}$ is exact since $\pi$ is a good quotient (see Lemma \ref{lem:multiple}) and $X\quot G$ is affine).

To make \eqref{eq:ext_cond3} true we choose $l,(m_i)_i$ in such a way that
\[
m_1\ll l+m_1\ll m_2\ll m_2+l\ll m_3\ll \cdots
\]
Then in \eqref{eq:ext_cond3} we have $\Lscr_u=\mu_i\otimes \chi^{a}\otimes \Oscr_{X^{ss}}$, $\Lscr_v=\mu_j\otimes\chi^{b}\otimes \Oscr_{X^{ss,\chi}}$
with $a\ll b$. Put $\Lscr=\pi_{s,\ast}(\chi\otimes \Oscr_{X^{ss,\chi}})$. By our choice of $\chi$, $\Lscr$ is ample on $X^{ss,\chi}\quot G$
 and $\pi^{s,\ast}\Lscr=\chi\otimes \Oscr_{X^{ss,\chi}}$ by Lemma \ref{lem:multiple}. Using the projection formula we have
\begin{align*}
\RHom_{X^{ss,\chi}/G}(\Lscr_u,\Lscr_v)&=R\Gamma(X^{ss,\chi}\quot G,\Lscr^{\otimes b-a}\otimes \pi_{s,\ast} (\mu_2\otimes \mu^{-1}_1\otimes \Oscr_{X^{ss,\chi}}))\\
&=\Gamma(X^{ss,\chi}\quot G,\Lscr^{\otimes b-a}\otimes \pi_{s,\ast} (\mu_2\otimes \mu^{-1}_1\otimes \Oscr_{X^{ss,\chi}}))
\end{align*}
where in the second line we use that $\Lscr$ is ample and $b-a\gg 0$.
\end{proof}
\begin{proof}[Proof of Theorem \ref{th:dim1}]  If $E,F$  are objects in an abelian category $\Ascr$ such that the Yoneda extension  
$\Ext^1_{\Ascr}(E,F)$  is a finitely generated right $\Hom_\Ascr(E,E)$-module
with generators $c_1,\ldots,c_n$
then we define the corresponding semi-universal extension of $E$ and $F$ to be 
the middle term of the extension
\[
0\r F\r \bar{F} \r E^{\oplus n}\r 0
\]
corresponding to $(c_i)_i$.

Let $(\Lscr_u)_{u=1,\ldots,N}$ be as in Lemma \ref{lem:generator}. 
Using the conditions (\ref{eq:ext_cond1},\ref{eq:ext_cond2},\ref{eq:ext_cond3}) as in Lemma \ref{lem:generator}
we may construct the object $\Tscr$ by taking successive semi-universal extensions among the $(\Lscr_u)_u$. See \cite[Lemma 3.1]{HillePerling1} for details. In loc.\ cit.\ universal extensions are considered but the argument also works with semi-universal extensions. 
\end{proof}

\section{Combinatorial interpretation}\label{combi}
We let $X,G,\chi$
be as in Theorem \ref{th:dim1}, without a priori assuming that 
 the fibers of $\theta: X^{ss,\chi}\quot G\r X\quot G$ have dimension
$\le 1$.
\begin{proposition}
\label{prop:fibers}
Assume
$X=W^\vee$ for a $G$-representation $W$.  Then if
\begin{equation}
\label{eq:comb}
\dim \hbox{{$X^u$}}-\dim G\le 1
\end{equation}
 the fibers of $\theta$ have dimension $\le 1$.
\end{proposition}
\begin{proof}
We refer to the diagram \eqref{eq:maps}.
By semi-continuity it is sufficient to bound the dimension of $\theta^{-1}(\bar{0})$, where $\bar{0}=\pi(0)$. Now $\theta^{-1}(\bar{0})=\pi(\tilde{\theta}^{-1}(\gamma^{-1}(\bar{0})))$. 
Since the fibers of $\pi$ have constant dimension $\dim G$ we deduce
\begin{multline*}
\dim \theta^{-1}(\bar{0})=\dim(\gamma^{-1}(\bar{0})\cap X^{ss,\chi})-\dim G\\
=\dim  (X^u\cap X^{ss,\chi})-\dim G\le \dim X^u-\dim G\qedhere.
\end{multline*}
\end{proof}

\begin{proof}[Proof of Proposition \ref{prop:main}]
We claim that we can apply Corollary \ref{cor:NCCR} to obtain an NCCR. 
 We need to verify that there exists $\chi\in X(G)$ such  $X=W^\vee$, $G$ $\chi$ satisfy assumptions of Theorem \ref{th:dim1}. 
By Remark \ref{rem:dm}, we may choose $\chi\in X(G)$ such that $X^{ss,\chi}/G$ is a DM stack. Moreover, the fibers of $\theta$ have dimension $\leq 1$ by Proposition \ref{prop:fibers}. 
\end{proof}

\begin{corollary} 
\label{cor:nccr2} Assume $X=W^\vee$ for a $G$-representation $W$.
If $W$ is generic and $\dim X\quot G=\dim X-\dim G\le 3$ then 
 the fibers of $\theta$ have dimension $\le 1$. In particular we recover the result by Broomhead that ``affine Gorenstein toric singularities of dimension $3$ have an NCCR''.
\end{corollary}
\begin{proof}
Let   $\beta_1,\ldots,\beta_d$ be the weights of $W$.
  The fact that $W$ is generic implies that for every $0\neq
  \lambda\in Y(G)$ we have that there are at least two $i$ such that
  $\langle \lambda,\beta_i\rangle >0$ (as otherwise $X-X^{\mathbf{s}}$ contains a codimension $1$ variety given by the vanishing of the  coordinate $x_i$ corresponding to the only $i$ for which $\la\lambda,\beta_i\ra>0$). Hence $\dim X^u\le \dim X-2$. 
Thus $\dim X^u-\dim G\le \dim X-\dim G -2\le 3-2=1$. In other words \eqref{eq:comb} holds.

For the last statement we employ Proposition \ref{prop:main} after noting that every affine toric variety $X$ can be written in the form $W^\vee\quot G$ for a generic $G$-representation~$W$. 
This is explained e.g. in \cite[\S11.6.1]{SVdB}. If $X$ is Gorenstein then of course so is $W^\vee\quot G$ and this is equivalent to $W$ being unimodular by \cite{Knop2}.
\end{proof}
Note that \eqref{eq:comb} may hold for higher dimensional $X\quot G$. 
Below we recall an example from \cite{SVdB} of a $4$-dimensional variety $X\quot G$ which does not have a toric NCCR. For this variety \eqref{eq:comb} is satisfied, and it therefore has a (non-toric) NCCR by Proposition \ref{prop:main} which we explicitly construct. 
\begin{example} 
\label{ex:toric}
Consider the example \cite[\S10.1]{SVdB}. Then we have that
$G=G^2_m$ is a two dimensional torus and (after the identifying $X(G)\cong \ZZ^2$) the weights $(\beta_i)_i$
of $W$
are given by  $(3,0),(1,1),(0,3),(-1,0),(-3,-3),(0,-1)$ (see Figure \ref{figure1}). We have $X^u=\bigcup_{1\leq i\leq 6}\{x_i=x_{i+1}=x_{i+2}=0\}$ with cyclic indices, hence $\dim X^u=3$ and 
moreover $W$ is generic and unimodular so that by Proposition \ref{prop:main}
$R=\Sym(W)^G=k[x_2x_4x_6,x_1x_3x_5,x_1x_4^3,x_3x_6^3,x_2^3x_5]\cong k[a, b, c, d, e]/(a^3 b -cde)$
 has an NCCR. However this NCCR is not toric which is the same as saying that it is not given by a module of covariants (a module of the form 
$M(U)=(U\otimes SW)^G$ for a finite dimensional $G$-representation $U$). 
In fact an NCCR given by a module of covariants  does not exist in this case as is shown 
in loc.\ cit.

\begin{figure}[H]{}\label{fig2}
\begin{tikzpicture}[scale=0.500000]
\path[draw,color=lightgray] (-4.000000,-4.000000) grid (4.000000,4.000000);
\path[draw,fill=vino,color=vino] (1.000000,1.000000) circle [radius=0.130000];
\node[inner sep=0pt,scale=0.5] at (1.25,1.2500000) {$2$};
\path[draw,fill=vino,color=vino] (-1.000000,0.000000) circle [radius=0.130000];
\node[inner sep=0pt,scale=0.5] at (-0.75,0.2500000) {$4$};
\path[draw,fill=vino,color=vino] (0.000000,-1.000000) circle [radius=0.130000];
\node[inner sep=0pt,scale=0.5] at (0.25,-0.7500000) {$6$};
\path[draw,fill=vino,color=vino] (-3.000000,-3.000000) circle [radius=0.130000];
\node[inner sep=0pt,scale=0.5] at (-2.75,-2.7500000) {$5$};
\path[draw,fill=vino,color=vino] (3.000000,0.000000) circle [radius=0.130000];
\node[inner sep=0pt,scale=0.5] at (3.25,0.2500000) {$1$};
\path[draw,fill=vino,color=vino] (0.000000,3.000000) circle [radius=0.130000];
\node[inner sep=0pt,scale=0.5] at (0.25,3.2500000) {$3$};
\path[draw,color=ruta] (0.000000,0.000000) circle [radius=0.180000];
\path[draw,color=ruta] (1.000000,0.000000) circle [radius=0.180000];
\path[draw,color=ruta] (-1.000000,0.000000) circle [radius=0.180000];
\path[draw,color=ruta] (-2.000000,0.000000) circle [radius=0.180000];
\path[draw,color=ruta] (0.000000,-1.000000) circle [radius=0.180000];
\path[draw,color=ruta] (-1.000000,-1.000000) circle [radius=0.180000];
\path[draw,color=ruta] (-2.000000,-1.000000) circle [radius=0.180000];
\path[draw,color=ruta] (1.000000,1.000000) circle [radius=0.180000];
\path[draw,color=ruta] (0.000000,1.000000) circle [radius=0.180000];
\path[draw,color=ruta] (-1.000000,1.000000) circle [radius=0.180000];
\path[draw,color=ruta] (1.000000,2.000000) circle [radius=0.180000];
\path[draw,color=ruta] (0.000000,2.000000) circle [radius=0.180000];
\path[draw,color=ruta,fill=ruta2] (2.000000,1.000000) circle [radius=0.180000];
\path[draw,color=black,thick] (0.000000,-0.300000) -- (0.000000,0.300000);
\path[draw,color=black,thick] (-0.300000,0.000000) -- (0.300000,0.000000);
\path[draw,fill=black,color=black] (-3.000000,-3.000000) circle [radius=0.0700];
\path[draw,fill=black,color=black] (-3.000000,-2.000000) circle [radius=0.0700];
\path[draw,fill=black,color=black] (-3.000000,-1.000000) circle [radius=0.0700];
\path[draw,fill=black,color=black] (-3.000000,0.000000) circle [radius=0.0700];
\path[draw,fill=black,color=black] (-2.000000,-3.000000) circle [radius=0.0700];
\path[draw,fill=black,color=black] (-2.000000,-2.000000) circle [radius=0.0700];
\path[draw,fill=black,color=black] (-2.000000,-1.000000) circle [radius=0.0700];
\path[draw,fill=black,color=black] (-2.000000,0.000000) circle [radius=0.0700];
\path[draw,fill=black,color=black] (-2.000000,1.000000) circle [radius=0.0700];
\path[draw,fill=black,color=black] (-1.000000,-3.000000) circle [radius=0.0700];
\path[draw,fill=black,color=black] (-1.000000,-2.000000) circle [radius=0.0700];
\path[draw,fill=black,color=black] (-1.000000,-1.000000) circle [radius=0.0700];
\path[draw,fill=black,color=black] (-1.000000,0.000000) circle [radius=0.0700];
\path[draw,fill=black,color=black] (-1.000000,1.000000) circle [radius=0.0700];
\path[draw,fill=black,color=black] (-1.000000,2.000000) circle [radius=0.0700];
\path[draw,fill=black,color=black] (0.000000,-3.000000) circle [radius=0.0700];
\path[draw,fill=black,color=black] (0.000000,-2.000000) circle [radius=0.0700];
\path[draw,fill=black,color=black] (0.000000,-1.000000) circle [radius=0.0700];
\path[draw,fill=black,color=black] (0.000000,0.000000) circle [radius=0.0700];
\path[draw,fill=black,color=black] (0.000000,1.000000) circle [radius=0.0700];
\path[draw,fill=black,color=black] (0.000000,2.000000) circle [radius=0.0700];
\path[draw,fill=black,color=black] (0.000000,3.000000) circle [radius=0.0700];
\path[draw,fill=black,color=black] (1.000000,-2.000000) circle [radius=0.0700];
\path[draw,fill=black,color=black] (1.000000,-1.000000) circle [radius=0.0700];
\path[draw,fill=black,color=black] (1.000000,0.000000) circle [radius=0.0700];
\path[draw,fill=black,color=black] (1.000000,1.000000) circle [radius=0.0700];
\path[draw,fill=black,color=black] (1.000000,2.000000) circle [radius=0.0700];
\path[draw,fill=black,color=black] (1.000000,3.000000) circle [radius=0.0700];
\path[draw,fill=black,color=black] (2.000000,-1.000000) circle [radius=0.0700];
\path[draw,fill=black,color=black] (2.000000,0.000000) circle [radius=0.0700];
\path[draw,fill=black,color=black] (2.000000,1.000000) circle [radius=0.0700];
\path[draw,fill=black,color=black] (2.000000,2.000000) circle [radius=0.0700];
\path[draw,fill=black,color=black] (2.000000,3.000000) circle [radius=0.0700];
\path[draw,fill=black,color=black] (3.000000,0.000000) circle [radius=0.0700];
\path[draw,fill=black,color=black] (3.000000,1.000000) circle [radius=0.0700];
\path[draw,fill=black,color=black] (3.000000,2.000000) circle [radius=0.0700000];
\path[draw,fill=black,color=black] (3.000000,3.000000) circle [radius=0.0700000];
\node[inner sep=0pt,scale=0.5] at (1.25,-1.7500000) {$\chi$};
\end{tikzpicture}
\caption{\textcolor{vino}{\large${\bullet}$}$^i$ weights
, \textcolor{black}{\tiny$\bullet$} CM weights, \textcolor{ruta}{\LARGE$\circ$}\textcolor{ruta2}{\LARGE$\bullet$} $\Lscr$}
\label{figure1}
\end{figure}

We will now describe the construction of  an explicit NCCR for this example. We have not literally followed the proof of Proposition \ref{prop:main} which appeared
computationally too expensive. 
Instead we  obtain an NCCR using a similar but more adhoc procedure.

First we give some heuristic motivation for the construction. Assuming an appropriately strengthened version of the Bondal-Orlov conjecture asserting 
that all (stacky) commutative and non-commutative crepant 
resolutions are derived equivalent \cite{BonOrl,IyamaWemyssNCBO,VdB32} the number of indecomposable summands of the module defining a  non-commutative crepant resolution that we need 
is given by 
the rank of $K_0$ of a (stacky) crepant commutative resolution of $\Spec R$ (since $K_0$ is invariant under  derived equivalence).

It is easy to verify that $\Spec R$ as a (singular) toric variety
corresponds to the fan given by the cone over a $3$-dimensional polytope $P$ shown in
Figure \ref{fig1}. The volume of this polytope equals $13/6$,
therefore the rank of $K_0$ of the stacky crepant resolution of $\Spec R$,
corresponding to a triangulation of $P$, is $13$ (see Theorem
\ref{BorisovHorja}).
\begin{figure}[h]{}
\centering
\includegraphics[width=6cm]{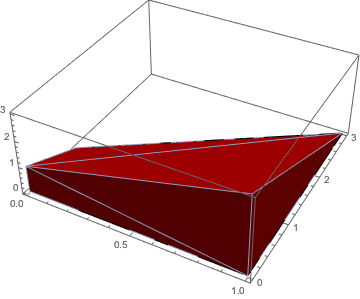}
\caption{}
\label{fig1}
\end{figure}

After these heuristics we describe  the actual construction. 
Let $\Lscr$\footnote{This notation is in accordance with notation in \cite[\S11]{SVdB5} which we will refer to in the sequel. It should not be confused with the notation for line bundles used in the previous sections.} be given by weights corresponding to encircled  dots in Figure \ref{fig2}
and let $\Lscr'=\Lscr\setminus \{(2,1)\}$. 
We write $M(\mu)=(\mu\otimes SW)^G$. 
The endomorphism ring  $\End_R(\oplus_{\mu\in \Lscr'}M(\mu))$ 
 is Cohen-Macaulay (see \cite[Example 10.1]{SVdB}). 
Since $|\Lscr'|=12$  
we expect to need  a single additional indecomposable $R$-module $K$ such that 
$\Lambda=\End_R(\oplus_{\mu\in \Lscr'}M(\mu)\oplus K)$ is an NCCR. 
By loc. cit. $K$ cannot be a module of covariants. 

We define $K$ by the exact sequence
\begin{equation}\label{covses}
\xymatrix{
0\r K\r M(0,-1)\oplus M(1,1)\oplus M(-1,1)\xrightarrow{\psi} M(2,1)\r 0, }
\end{equation}
where $\psi(r_1,r_2,r_3)=r_1 d+r_2 a+ r_3 ab/c$.
(Note that $M(0,-1)\cong (a,e)$, $M(1,1)\cong (a,d)$, $M(-1,1)\cong (a^2,ac,cd)$, $M(2,1)\cong (a^2,ad,de)$.) 
Considering $M(\mu)$ as subsets of $\Sym(W)$ 
we can write $\psi(r_1,r_2,r_3)=r_1 x_2x_5+r_2 x_4+r_3 x_3x_5$.

It is easy to check that $\Lambda$ is a Cohen-Macaulay $R$-module (using e.g.\ Macaulay2), suggesting that it might be an NCCR of $R$. Below we verify this fact by constructing an appropriate tilting
 bundle on a particular stacky resolution $X^{ss,\chi}/G$ of $\Spec R$.

Let $\chi=(1,-2)$. We claim that $\Escr=\bigoplus_{\mu\in \Lscr}\mu\otimes \Oscr_{X^{ss,\chi}}$ generates $D(X^{ss,\chi}/G)$.  
One can use  a similar algorithm as in the proof of \cite[Theorem 1.5.1]{SVdB}. We refer to  \cite[\S11.1-3]{SVdB} for some unexplained notation.   
In loc. cit. the complexes $C_{\lambda,\mu}$ with  cohomology supported on $X^{\lambda,\geq 0}$ relate projectives $P_\nu$, $\nu\in X(T)$, in $D(X/G)$. 
Thus, if $\la\chi,\lambda\ra<0$ then $C_{\lambda,\mu}$ is exact when restricted to $X^{ss,\chi}/G$  
(recall that $X^{ss,\chi}$ consists of $x\in X$ such that if $\lambda\in Y(T)$ is such that $\lim_{t\to 0}\lambda(t)x$ exists then $\la\lambda,\chi\ra\geq 0$ which is equivalent to saying that $-\chi$ is in the cone generated by $(\beta_i)_{x_i\neq 0}$). 
Assume that $\tilde{\Lscr}\subset X(T)$ is such that $\nu\otimes \Oscr_{X^{ss,\chi}}$, $\nu\in \tilde{\Lscr}$, belong to the subcategory of $D(X^{ss,\chi}/G)$  generated by $\Escr$  (e.g.\ $\tilde{\Lscr}=\Lscr$).
Then we may enlarge $\tilde{\Lscr}$ by $\nu\in X(T)$ if for some $\la\lambda,\chi\ra<0$ all components except for $\nu\otimes\Oscr_{X^{ss,\chi}}$ of either of the complexes $C_{\lambda,\nu}$, $C_{\lambda,\nu-\sum_{\la\lambda,\beta_i\ra>0} \beta_i}$   are of the form $\mu\otimes \Oscr_{X^{ss,\chi}/G}$ for $\mu\in \tilde{\Lscr}$. 
Note that if $\tilde{\Lscr}$ contains $X(T)\cap \Sigma$, then we may enlarge $\tilde{\Lscr}$ to $X(T)$. (See also the proof of \cite[Theorem 3.2]{HLSam}.)

In our example we may easily verify by hand (or by computer, cf. \cite[Remark 11.3.2]{SVdB}) that we can indeed enlarge $\Lscr$ to $\Sigma\cap X(T)$ (where in this case $\Sigma\cap X(T)$ is given by weights corresponding to  black dots in the above picture), and therefore $\Escr$  generates $D(X^{ss,\chi}/G)$. 

Since the endomorphism ring  $\End_R(\oplus_{\mu\in \Lscr'}M(\mu))$ is Cohen-Macaulay, we have $\Ext^1_{X^{ss,\chi}/G}(\Escr',\Escr')=0$ for $\Escr'=\bigoplus_{\mu\in\Lscr'}\mu\otimes \Oscr_{X^{ss,\chi}}$  (see \cite[Corollary 3.3.2]{Vdb1}).

Denote $M=\{(0,-1),(1,1),(-1,1)\}\subset \Lscr'$. 
Let $\tilde \psi:\bigoplus_{\mu\in M} \mu\otimes \Oscr_{X^{ss,\chi}}\to \mu_{(2,1)}\otimes\Oscr_{X^{ss,\chi}}$ be the lift of $\psi$ to $X^{ss,\chi}/G$, and let $\Kscr$ be the lift of $K$ (see the proof of Corollary \ref{cor:dim1}). 
We claim that \eqref{covses} induces an exact sequence 
\begin{equation}\label{stackses}
0\r \Kscr\r \bigoplus_{\mu\in M} \mu\otimes \Oscr_{X^{ss,\chi}}\xrightarrow{\tilde{\psi}} \mu_{(2,1)}\otimes\Oscr_{X^{ss,\chi}}\r 0.
\end{equation}
Since $\tilde\psi$ is a restriction of the map $\Psi:\oplus_{\mu\in M}\mu\otimes \Oscr_X\to \mu_{(2,1)}\otimes \Oscr_X$, induced from $\psi$, we need to check that the cokernel $\Nscr$ of this map has support in the complement of $X^{ss,\chi}$. 
The support of the cokernel is defined by the ideal $(x_2x_5,x_4,x_3x_5)$. Let $x=(x_1,\dots,x_6)$ belong to the support. Then either $x_2=x_3=x_4=0$ or $x_4=x_5=0$. Since $-\chi$ does not lie in the cone generated  by neither $\beta_1,\beta_5,\beta_6$ nor $\beta_1,\beta_2,\beta_3,\beta_6$, $x$ does not belong to $X^{ss,\chi}$.

Moreover, any map from $\bigoplus_{\mu\in\Lscr'}\mu \otimes \Oscr_X$ to $\mu_{(2,1)}\otimes \Oscr_X$ factors through $\Psi$, since its image is zero in $\Nscr$ which easily follows from the fact that $\Lscr'$ does not intersect the semigroups generated by $\beta_1,\beta_5,\beta_6$ and $\beta_1,\beta_2,\beta_3,\beta_6$, resp., shifted by $\mu_{(2,1)}$. 
Therefore, employing again the proof of Corollary \ref{cor:dim1},  
the map $\Hom_{X^{ss,\chi}/G}(\mu_i\otimes \Oscr_{X^{ss,\chi}},\bigoplus_{\mu\in M}\mu\otimes \Oscr_{X^{ss,\chi}})\to \Hom_{X^{ss,\chi}/G}(\mu_i\otimes \Oscr_{X^{ss,\chi}},\mu_{(2,1)}\otimes \Oscr_{X^{ss,\chi}})$ induced from \eqref{stackses} is surjective.
Thus, $\Ext^1_{X^{ss,\chi}/G}(\mu_i\otimes \Oscr_{X^{ss,\chi}},\Kscr)=0$.

Applying $\Hom_{X^{ss,\chi}/G}(-,\mu_i\otimes \Oscr_{X^{ss,\chi}})$ and 
$\Hom_{X^{ss,\chi}/G}(-,\Kscr)$ to \eqref{stackses} further implies that 
$\Ext^1_{X^{ss,\chi}/G}(\Kscr,\mu_i\otimes \Oscr_{X^{ss,\chi}})=0$ and 
$\Ext^1_{X^{ss,\chi}/G}(\Kscr,\Kscr)=0$. 

Since $\Escr$ generates $D(X^{ss,\chi}/G)$, the same holds for $\Fscr=\bigoplus_{\mu\in \Lscr'}\mu\otimes\Oscr_{X^{ss,\chi}}\bigoplus\Kscr$ by \eqref{stackses},  
and we moreover have $\Ext^1_{X^{ss,\chi}/G}(\Fscr,\Fscr)=0$.   
Thus, $\End_R(\bigoplus_{\mu\in \Lscr'}M(\mu)\bigoplus K)$ is an NCCR of $R$ by Corollary \ref{cor:NCCR}.
\end{example}

\begin{remark}
The discussion on the ``universality'' of $\Psi$ in fact implies that $\psi$ in \eqref{covses} is the minimal $\operatorname{add} (\oplus_{\mu\in\Lscr'}M(\mu))$-approximation of $M(2,1)$ 
 in the sense that every map $M(\mu)\r M(2,1)$ for $\mu\in \Lscr'$ factors through~$\psi$.
\end{remark}

\begin{remark}
Let $S=k[a,b,c,d,e]$. 
The module $K$ introduced in the above example may also be described 
by a matrix factorization $(d_0,d_1)$ of $f=a^3b-cde$:
\[
d_0=\left(
\begin{array}{cccc}
ab&0&ce&0\\
0&ab&-ac&-cd\\
-d&0&-a^2&0\\
-a&-e&0&a^2
\end{array}
\right),\quad\quad
d_1=\left(
\begin{array}{cccc}
a^2&0&ce&0\\
0&a^2&-ac&cd\\
-d&0&-ab&0\\
a&e&0&ab
\end{array}
\right),
\]
where $d_0,d_1:S^4\to S^4$ and $K=\coker(d_0)$.
\end{remark}

\appendix
\section{Grothendieck group of a toric DM stack}
Here we recall some results about the Grothendieck group of a toric DM stack. 
We mainly follow \cite{BorisovHorja}.

Let $\Sigma$ be a fan, refining a cone over an $n-1$-dimensional convex lattice polyhedron $P\times \{1\}$. 
Let ${\bf \Sigma}$ be a stacky fan $(\Sigma,(v_i)_{i=1}^l)$, where $v_i\in \ZZ^{n-1}\times \{1\}$ define $1$-dimensional cones in $\Sigma$. 
We denote by $P_{\bf \Sigma}$ (resp. $P_\Sigma$) the corresponding toric DM stack (resp. toric variety). 
Note that $P_{\bf \Sigma}$ (resp. $P_\Sigma$) equals $Y^{ss,\chi}/ G$  (resp. $Y^{ss,\chi}\quot G $) for an action of $G\subset {k^*}^l$ on $Y=k^n$ via characters determined by the images of the generators $e_i\in \ZZ^l$ ($e_i$ denotes the $i$-th generator) in $\ZZ^l/ \rho(M)\cong X(G)$ ($\rho:m\mapsto (\la m,v_i\ra)_i$) and a generic $\chi\in X(G)$ (see \cite[Section 2]{BorisovHorja}, \cite[Theorem 15.1.10]{CoxLittleSchenck}). 

Let $\mu_i=\bar{e_i}\in X(G)$.   
We denote by $R_i$ the class of the invertible sheaf $\mu_i\otimes \Oscr_{Y^{ss,\chi}}$ in $K_0(P_{\bf \Sigma})$. 

\begin{theorem}\cite{BorisovHorja}\label{BorisovHorja}
Let $P_{\bf\Sigma}$ be a toric DM stack. 
Let $B$ be the quotient of the Laurent polynomial ring $\ZZ[x_1,x_1^{-1},\dots, x_l,x_l^{-1}]$ by the ideal generated by the relations 
\begin{itemize}
\item 
$\prod_{i=1}^l x_i^{\la m,v_i\ra}=1$ for all $m\in M$,
\item
$\prod_{i\in I}(1-x_i)=0$ for any set $I\subseteq \{1,\dots,l\}$ such that $v_i$, $i\in I$, are not contained in any cone of $\Sigma$.
\end{itemize}
Then the map $\phi:B\to K_0(P_{\bf\Sigma})$ which sends $x_i$ to $R_i$
is an isomorphism. If $P_\Sigma$ is a triangulation of a cone over a polyhedron $P$, then ${\rm rk} K_0(P_{\bf \Sigma})=(n-1)!{\rm Vol} (P)$.
\end{theorem}

\begin{proof}
First part follows by \cite[Theorem 4.10]{BorisovHorja}, while the last statement follows from \cite[Remark 3.11, Theorem 5.3]{BorisovHorja}. 
Indeed, we only need to show that $(1-t)^l\sum_{n\in N\cap \sigma}t^{{\rm deg}(n)}$ evaluated at $1$ equals $(n-1)!\Vol (P)$. 

Note that $N\cap \s=\bigcup_{d\in \NN} dP\times \{d\}$, and ${\rm deg}(n)=d$ for $n\in dP\times \{d\}$. Moreover, the number of lattice points in $dP\times \{d\}$ equals ${\rm Ehr}_P(d)$, where ${\rm Ehr}$ denotes the Ehrhart polynomial (see e.g. \cite[Theorem 9.4.2]{CoxLittleSchenck}). 
 Since the degree of $\Ehr$ is $n-1$ (as $P$ is $n-1$-dimensional) and its leading coefficient equals $(n-1)!\Vol(P)$ (see e.g. \cite[Exercise 9.4.7]{CoxLittleSchenck}) we obtain that the above sum evaluated at $1$ equals $(n-1)!\Vol(P)$. 
\end{proof}

\bibliographystyle{amsalpha}
\bibliography{nccr}
\end{document}